\documentclass{article}

\makeatletter
\newsavebox\myboxA
\newsavebox\myboxB
\newlength\mylenA

\newcommand*\xoverline[2][0.75]{%
	\sbox{\myboxA}{$\m@th#2$}%
	\setbox\myboxB\null% Phantom box
	\ht\myboxB=\ht\myboxA%
	\dp\myboxB=\dp\myboxA%
	\wd\myboxB=#1\wd\myboxA% Scale phantom
	\sbox\myboxB{$\m@th\overline{\copy\myboxB}$}%  Overlined phantom
	\setlength\mylenA{\the\wd\myboxA}%   calc width diff
	\addtolength\mylenA{-\the\wd\myboxB}%
	\ifdim\wd\myboxB<\wd\myboxA%
	\rlap{\hskip 0.5\mylenA\usebox\myboxB}{\usebox\myboxA}%
	\else
	\hskip -0.5\mylenA\rlap{\usebox\myboxA}{\hskip 0.5\mylenA\usebox\myboxB}%
	\fi}
\makeatother
\usepackage{amsmath,amsthm,amsfonts,amssymb,amscd} %AMS宏包
\usepackage{bm} %加粗字符
\usepackage[top=2.54cm,bottom=2.54cm,left=3.17cm,right=3.17cm]{geometry}  %设置页边距
\usepackage{esint,mathtools,bbm} %特殊字符
\usepackage{tikz}
\usetikzlibrary{backgrounds}
\usepackage{float} %固定图片
\usepackage{algorithm,algorithmicx,algpseudocode}  %伪代码
\usepackage{booktabs} %三线表
\algnewcommand{\algorithmicgoto}{\textbf{go to}}%
\algnewcommand{\Goto}[1]{\algorithmicgoto~\ref{#1}}%
\usepackage{cleveref}
\crefformat{equation}{(#2#1#3)}

\usepackage{xcolor}
\definecolor{mycolor1}{RGB}{237,227,135}
\definecolor{mycolor2}{RGB}{237,237,237}
\definecolor{mycolor3}{RGB}{59,32,12}
\definecolor{mycolor4}{RGB}{222,129,0}
\definecolor{mycolor5}{RGB}{204,252,98}

\newtheorem{theorem}{Theorem}[section]
\newtheorem{lemma}[theorem]{Lemma}
\newtheorem{proposition}[theorem]{Proposition}

\theoremstyle{definition}
\newtheorem{definition}[theorem]{Definition}

\newtheorem{corollary}[theorem]{Corollary}

\theoremstyle{remark}
\newtheorem{remark}[theorem]{Remark}

\numberwithin{equation}{section}

\newcommand{\Z}{\mathbb{Z}}
\newcommand{\R}{\mathbb{R}}

\newcommand{\Brackets}[1]{\left( #1 \right)}
\newcommand{\SquareBrackets}[1]{\left[ #1\right]}
\newcommand{\Braces}[1]{\left\{ #1\right\}}
\newcommand{\FuncAction}[2]{\left\langle #1,#2 \right\rangle}
\newcommand{\Norm}[1]{\left\lVert #1 \right\rVert}
\newcommand{\Seminorm}[1]{\left\lvert #1 \right\rvert}
\newcommand{\dx}{\mathrm{d}}
\newcommand{\overbar}[1]{\mkern 1.5mu\overline{\mkern-1.5mu#1\mkern-1.5mu}\mkern 1.5mu} %闭集

\newcommand{\Div}{\mathrm{div}}
\newcommand{\Supp}{\mathrm{supp}}
\newcommand{\Sop}{\mathbf{S}}

\title{A First-order Two-scale Analysis for Contact Problems with Small Periodic Configurations}
\author{Changqing Ye\footnote{LSEC, ICMSEC, Academy of Mathematics and Systems Science, Chinese Academy of Sciences, Beijing 100190, China, and School of Mathematical Sciences, University of Chinese Academy of Sciences, Beijing 100049, China}, Junzhi Cui\footnote{LSEC, ICMSEC, Academy of Mathematics and Systems Science, Chinese Academy of Sciences, Beijing 100190, China}}
\date{\today \footnote{Version 0.6}}

\begin{document}
\maketitle
\begin{abstract}
	This paper is devoted to studying a type of contact problems modeled by hemivariational inequalities with small periodic coefficients appearing in PDEs, and the PDEs we considered are linear, second order and uniformly elliptic. Under the assumptions, it is proved that the original problem can be homogenized, and the solution weakly converges. We derive an $O(\epsilon^{1/2})$ estimation which is pivotal in building the computational framework. We also show that Robin problems--- a special case of contact problems, it leads to an $O(\epsilon)$ estimation in $L^2$ norm. Our computational framework is based on finite element methods, and the numerical analysis is given, together with experiments to convince the estimation.
\end{abstract}

\section{Introduction}
% Small periodic setting for composite material, history, ...
In composite material design and performance optimization, the controlling PDEs within the models frequently involve small periodic coefficients (e.g., \cite{Yang2017,Dong2018,Liu2013,Liu2015}). For those problems, the periodic homogenization theory is the basis, and many PDE experts contributed considerable works to build this theory. For examples: qualitative results such as asymptotic expansion \cite{Bensoussan1978,Oleuinik1992}, H-convergence \cite{Tartar2010}, $\Gamma$-convergence \cite{DalMaso1993}, and two-scale convergence \cite{Allaire1992}; regularity results such as compactness property investigated by M. Avellaneda and Lin in \cite{Avellaneda1987}, and recently a thorough study for Neumann boundary condition by Shen et al. \cite{Kenig2013,Shen2016}. As a model problem for the multiscale phenomenon, it also attracts great attention among scientific computation community. Due to the high oscillation emerging in the solution, classical computational method, such as finite element methods (FEM) can not reveal the fine scale information. Because of strong practical background, several modern multiscale computation methods have been developed since 1990s. We can classify those as three groups in methodology: modify FEM piecewise polynomial basis to enhance the expression ability, such as MsFEM \cite{Hou1999,Efendiev2000,Chen2002} and LOD \cite{Malqvist2014,Henning2014}; utilize scale separation property to decompose original solution into coarsen and fine parts, VMM \cite{Hughes1995,Brezzi1997,Larson2007}; improve the accuracy of homogenized solution by involving smallscale information, HMM \cite{Weinan2005,Abdulle2012}. In those methods Dirichlet problem is chosen when conducting numerical experiments and error analysis, while Robin problem or more general contact problems are scarcely investigated.

% Hemivariational inequality history... numerical method, ...
The notion of hemivariational inequalities was first introduced by Panagiotopoulos in the early 1980s \cite{Panagiotopoulos1983}. Since then, hemivariational inequalities receive broad applications in nonsmooth mechanics, contact mechanics, physics, and economics \cite{Naniewicz1995,Panagiotopoulos1993,Migorski2013,Goeleven2003}. In this paper, we focus on boundary hemivariational inequality problems, which originate from the mathematical model of elastic contact system. To solve this kind of hemivariational inequalities, a finite element method had been implemented \cite{Haslinger1999} while the thorough numerical analysis has not been established until recent. In \cite{Han2017}, Han et al. derived a C\'{e}a's inequality in an abstract framework, and figure out the influence of solution's regularity to the numerical computation. To our knowledge, proper assumptions will balance the solvability and generality of mathematical models, and this is extremely important in nonlinear problems. Hence, we adopt the assumptions in \cite{Han2017} to prove our main results.

% Some words
To our knowledge the study on hemivariational inequalities with the coefficients setting in small periodic configurations is few. The homogenization result could be found in \cite{Liu2008}. However, the result or H-convergence property does not provide a priori convergence rate which is pivotal in numerical analysis. It explains why we need to build an $O(\epsilon^{1/2})$ estimation. In the following sections, we set model problem in a scalar form merely for the simplicity of symbols, and the extension to elastic system will be straight.
  
% Main purpose, Liu's paper
The rest of paper is organized as follows. In \cref{sec:preliminaries} we introduce notations, review some preliminary materials including generalized directional derivative, and state the model problem and assumptions for later proof. In \cref{sec:homogenization}, firstly, we prove a uniform bound for solutions which is missing in \cite{Liu2008}, and we think it is indispensable. Then we apply div-curl lemma to prove the homogenization result, the proof will also be provided for the self-containing. We give an $O(\epsilon^{1/2})$ estimation for first order asymptotic expansion in \cref{sec:estimation}, and the insight most comes from \cite{Shen2016}. We discuss Robin problem in \cref{sec:robin}, and show that with the duality technique from \cite{Shen2016} the difference between original and homogenized solutions in $L^2$ norm is $O(\epsilon)$. A computational framework based on finite element methods will be presented in \cref{sec:computation} together with its numerical analysis.  Experiments are reported in \cref{sec:experiments}, and the results are in good agreement with predicted estimation.

\section{Preliminaries}
\label{sec:preliminaries}
Generally, when $X$ is used, it denotes a real Banach space with its norm as $\Norm{\cdot}_X$, $X^*$ as its topological dual, $\FuncAction{\cdot}{\cdot}_{X^*\times X}$ as duality pairing. Without confusion, we omit the subscript and simply write $\FuncAction{\cdot}{\cdot}$. Weak convergence is indicated by $\rightharpoonup$. Given two normed spaces $X$ and $Y$, $\mathcal{L}(X, Y)$ is the space of all linear continuous operators from $X$ to $Y$. 

Notation $d$ is always used as the space dimension. In the full text, the Einstein summation convention is adopted, means summing repeated indexes from $1$ to $d$. Without specification, $\Omega$ is a domain (open and bounded set in $\R^d$) with Lipschitz boundary $\Gamma$, and denote $\bm{n}$ as the outward unit normal to $\Gamma$. The Sobolev spaces $W^{k,p}$ and $H^{k}$ are defined as usual (see \cite{Brenner2008}) and we abbreviate the norm and seminorm of Sobolev space $H^k(\Omega)$ as $\Norm{ \cdot}_{k,\Omega}$ and $\Seminorm{\cdot}_{k,\Omega}$. 

To specify conditions respectively on the different parts of boundary, we rewrite $\Gamma= \xoverline{\Gamma_D} \cup \xoverline{\Gamma_N} \cup \xoverline{\Gamma_C}$, $\Gamma_D$, $\Gamma_N$ and $\Gamma_C$ are open according to the inheriting topology on $\Gamma$ and disjoint with each other, $\Gamma_C \neq \varnothing$ and $\Gamma_D \neq \varnothing$ without specification. We mainly concern functional space $V$ which its functions $u \in H^1(\Omega)$ and vanishing on $\Gamma_D$ in the sense of trace, and one can easily check that $V$ is Hilbertian, and the norm can be legally set as $\Norm{\cdot}_V=\Seminorm{\cdot}_{1,\Omega}$ when $\Gamma_D \neq \varnothing$. Following the notations in \cite{Han2017}, we denote $V_j=L^2(\Gamma_C)$ as the main space for hemivariational inequality and $\gamma_j \in \mathcal{L}(V,V_j)$ as trace operator from $V$ to $V_j$. We point here that the split $\Gamma= \xoverline{\Gamma_D} \cup \xoverline{\Gamma_N} \cup \xoverline{\Gamma_C}$ must be regular enough to guarantee that $\gamma_j$ is compact, and normally it is true because that fractional Sobolev space $H^{1/2}(\Gamma_C)$ is compactly embedded into $L^2(\Gamma_C)$ (refer \cite{Nezza2012} for more details).

To describe the periodic structure, we denote $Q=(-1/2, 1/2)^d$ as a representative cell, and call a function $f$ 1-periodic, it means:
\[
f(\bm{x}+\bm{z}) = f(\bm{x}) ~~~~ \forall \bm{x} \in \R^d \text{ and } \forall \bm{z} \in \Z^d ,
\] 
and we also use a superscript $\epsilon$ for $f(\bm{x})$ to represent scaling $f^\epsilon(\bm{x})\coloneqq f(\bm{x}/\epsilon)$ if $f$ is 1-periodic. $H^k_\sharp(Q)$ or $W^{k,p}_\sharp(Q)$ with "$\sharp$" means this functional space is the completion of smooth 1-periodic functions with respect to the $H^k(Q)$ or $W^{k,p}(Q)$ norm. We have a fundamental theorem for $f^\epsilon$:
\begin{theorem}[see \cite{Cioranescu1999} Theorem 2.6]\label{thm:oscillating converge}
	Let $1\leq p < \infty$,  $\forall f \in L^p_\sharp(Q)$, then $f^\epsilon \rightharpoonup \mathcal{M}_Q f= 1/\Seminorm{Q}\int_Q f$ in $L^p(\omega)$. Here $\omega$ is an arbitrary bounded open subset in $\R^d$. 
\end{theorem}

It is customary to write $C$ as a positive constant, and $C(p_1,\cdots,p_n)$ indicates that $C$ depends on $p_1,\cdots,p_n$. 

Then we introduce (Clarke) generalized directional derivative and subdifferential (see \cite{Denkowski2003}).
\begin{definition}
	Let $\varphi: X \rightarrow \R$ be a locally Lipschitz function. For $x, h \in X$, the \emph{generalized directional derivative} of $\varphi$ at $x$ along the direction $h$, denoted by $\varphi^0(x; h)$ is defined by
	\[
	\varphi^0(x; h) \coloneqq \limsup_{y\rightarrow x, \lambda \downarrow 0} \frac{\varphi(y+\lambda h)-\varphi(y)}{\lambda} = \inf_{\epsilon, \delta>0} \sup_{\substack{\Norm{x-y}_X<\epsilon \\ 0<\lambda< \delta}} \frac{\varphi(y+\lambda h)-\varphi(y)}{\lambda} .
	\]
	The \emph{generalized subdifferential} of $\varphi$ at $x \in X$, is the nonempty set $\partial \varphi (x) \subset X^*$ defined by
	\[
	\partial \varphi(x) \coloneqq \{ x^* \in X^*: \FuncAction{x^*}{h} \leq \varphi^0(x; h), \forall h \in X \}.
	\] 
\end{definition}

From now on, the matrix function 
\[
A^\epsilon(\bm{x})=\SquareBrackets{A^\epsilon_{ij}(\bm{x})}_{1\leq i, j \leq d}=A(\bm{x}/\epsilon)=\SquareBrackets{A_{ij}(\frac{\bm{x}}{\epsilon})}_{1\leq i, j \leq d}
\]
serves as the coefficients in our PDE model. The scale parameter $\epsilon \ll 1$. Also provide $f \in V^*$, $g \in L^2(\Gamma_N)$, and $j: V_j \rightarrow \R$ is a locally Lipschitz function. Now we can formulate our contact problem, 
\begin{equation}\label{eq:contact problem}
\left\{
\begin{aligned}
-\Div(A^\epsilon(\bm{x})\nabla u_\epsilon) = f ~~~~ &\text{in} ~~ \Omega \\
u_\epsilon = 0 ~~~~ &\text{on} ~~ \Gamma_D \\
\bm{n} \cdot A^\epsilon \nabla u_\epsilon = g ~~~~ &\text{on} ~~ \Gamma_N \\
-\bm{n} \cdot A^\epsilon \nabla u_\epsilon \in \partial j(\gamma_j u_\epsilon) ~~~~ &\text{on} ~~ \Gamma_C
\end{aligned}	
\right. ,
\end{equation}
and its hemivariational form:
\begin{equation}\label{eq:contact problem hemiform}
\left\{
\begin{aligned}
&\text{Find } u_\epsilon \in V, \text{ s.t. } \forall v \in V \\
& \int_\Omega A^\epsilon \nabla u_\epsilon \cdot \nabla v + j^0(\gamma_j u_\epsilon; \gamma_j v) \geq \FuncAction{f}{v}+\int_{\Gamma_N} gv 
\end{aligned}
\right. .
\end{equation}
We mention that $v \mapsto \int_{\Gamma_N} gv$ is a bounded functional on $V$ since $g \in L^2(\Gamma_N)$, one can rewrite $\FuncAction{\tilde{f}}{v}\coloneqq \FuncAction{f}{v}+\int_{\Gamma_N} gv$. It is clear that $\Norm{\tilde{f}}_{V^*} \leq \Norm{f}_{V^*}+c_k \Norm{g}_{0,\Gamma_N}$, where $c_k$ equals the trace operator norm from $V$ to $L^2(\Gamma_N)$.

To make this hemivariational form solvable, we need the following assumptions: 
\begin{itemize}
	\item[\textbf{A}:] The coefficient matrix $A(y)$ is symmetric and uniformly elliptic: 
	\begin{equation}\label{ass:A}
	\begin{aligned}
	& A_{ij}(\bm{y}) = A_{ji}(\bm{y}) \\
	& \kappa_1 \Seminorm{\xi}^2 \leq A_{ij}(\bm{y})\xi_i \xi_j \leq \kappa_2 \Seminorm{\xi}^2 ~~~~ \text{for a.e. } \bm{y} \in \R^d \text{ and } \xi \in \R^d
	\end{aligned} ~ .
	\end{equation}
	\item[\textbf{B}:] There exist constants $c_0,c_1,\alpha_j$, such that:
	\begin{equation}\label{ass:B1}
	\Norm{x^*}_{V_j^*} \leq c_0+c_1\Norm{x}_{V_j} \forall x \in V_j, \forall x^* \in \partial j(x) ,
	\end{equation}
	\begin{equation}\label{ass:B2}
	j^0(x_1; x_2-x_1)+j^0(x_2; x_1-x_2) \leq \alpha_j \Norm{x_1-x_2}_{V_j}^2 \forall x_1, x_2 \in V_j .
	\end{equation}
	\item[\textbf{C}:] Let $c_j=\Norm{\gamma_j}_{V\rightarrow V_j}$ be the operator norm, there exists that, 
	\begin{equation}\label{ass:C}
	\Delta \coloneqq \kappa_1 - \alpha_j c_j^2 > 0 .
	\end{equation}
\end{itemize}

We remark here that assumptions \textbf{B} and \textbf{C} follow \cite{Han2017}. Then by utilizing the framework constructed in \cite{Han2017} following theorem is obvious:
\begin{theorem}
	The solution of problem \cref{eq:contact problem hemiform} exists and is unique.
\end{theorem}
\begin{remark}
	The assumptions in former work \cite{Liu2008} are sightly different with \cite{Han2017}. For example, \cite{Liu2008} needs $j^0_N(x, r; -r) \leq d_N(1+\Seminorm{r})$ which is redundant in our case. Hence, we think the assumptions in \cite{Han2017} are more reasonable.
\end{remark}

We close this section by illustrating a specific contact problem. Here \cref{fig:Rubin} describes a domain $\Omega$ with its boundary $\Gamma$ composed by $\Gamma_D, \Gamma_N, \Gamma_{C'}, \Gamma_{C^{\prime\prime}}$. On $\Gamma_D$, we have $u=0$; On $\Gamma_N$, we have $g=\bm{n}\cdot A \nabla u$; On $\Gamma_{C'}$, a complete Robin condition is imposed, means $-\bm{n}\cdot A \nabla u = u$; while on $\Gamma_{C^{\prime\prime}}$, it is instead with a partial Robin condition $-\bm{n}\cdot A \nabla u = 0$ if $u<0$ and $-\bm{n}\cdot A \nabla u = u$ if $u\geq 0$, or write shortly as $-\bm{n}\cdot A \nabla u = u^+$. 

One can check in this problem $j^0(\gamma_ju; \gamma_j v)=\int_{\Gamma_{C'}}uv+\int_{\Gamma_{C^{\prime\prime}}}u^+v $. The original locally Lipschitz function $j(u)=1/2 \int_{\Gamma_{C'}}u^2+1/2\int_{\Gamma_{C^{\prime\prime}}} \Brackets{u^+}^2$, and $\partial j (\gamma_ju)=u$ on $\Gamma_{C'}$ and $\partial_j(\gamma_ju)=u^+$ on $\Gamma_{C^{\prime\prime}}$.  $\partial j, j^0$ coincide respectively with classical G\^{a}teaux and Fr\`{e}chet derivative of functional $j(u)$.

\begin{figure}[htbp]
			\centering
\begin{tikzpicture}
\draw (0, 0) rectangle (4, 4);
\foreach \m in {1,...,13}
{
	\draw[color=gray] (0.3*\m-0.05, 4) -- +(135:0.3);
	\draw[color=gray,->] (0, 0.3*\m-0.05) -- +(-0.3, 0);
	\draw[color=gray,->] (4, 0.3*\m-0.05) -- +(0.3, 0);	
}
\fill[mycolor3] (-0.5, -0.5) rectangle (2, 0);
\fill[mycolor4] (2, -0.5) rectangle (4.5, 0);
\node at (0.4, 2) {$\Gamma_N$};
\node at (3.6, 2) {$\Gamma_N$};
\node at (2, 3.6) {$\Gamma_D$};
\node at (1, 0.4) {$\Gamma_{C'}$};
\node at (3, 0.4) {$\Gamma_{C^{\prime\prime}}$};
\node at (2, 2) {$\Omega$};

\begin{scope}[on background layer]
\foreach \m in {1,...,8}
\foreach \n in {1,...,8}
{
	\fill[mycolor1] (0.5*\m-0.5, 0.5*\n-0.5) rectangle (0.5*\m, 0.5*\n);
	\fill[mycolor2] (0.5*\m-0.25, 0.5*\n-0.25) circle (0.15);
}

\end{scope}
\end{tikzpicture}
	\caption{A domain $\Omega$, its boundary is split into four parts $\Gamma_D, \Gamma_N, \Gamma_{C'}, \Gamma_{C^{\prime\prime}}$}\label{fig:Rubin}
\end{figure}
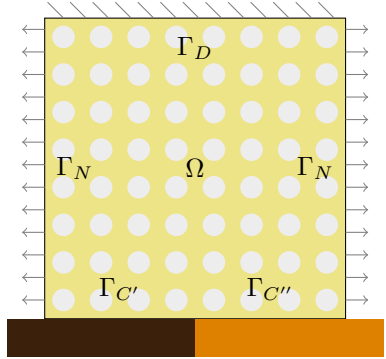

\section{Homogenization}
\label{sec:homogenization}
In this section, we will give the definition of correctors and present that the problem \cref{eq:contact problem hemiform} has a homogenized version, then we use div-curl lemma to prove $u_\epsilon$ can weak converge to the homogenized solution.

We denote $N_l(\bm{y})$ as correctors for $A^\epsilon$, which satisfy a group of PDEs with periodic boundary condition:
\begin{equation}\label{eq:correctors}
\left\{
\begin{aligned}
& -\Div\Brackets{A(\bm{y})\nabla N_l}=\Div (Ae_l)=\partial_i A_{il}(\bm{y}) ~~~~ \text{ in }Q\\
& N_l(\bm{y}) \in H^1_\sharp(Q) \text{ and } \int_Q N_l = 0
\end{aligned}
\right. .
\end{equation}
For correctors, we have $\Norm{N_l}_{1,Q} \leq C(\kappa_1, \kappa_2)$. The homogenized coefficients are defined by $\hat{A}_{il}=\fint_Q A_{il}+A_{ij}\partial_j N_l \dx \bm{y}$. The next lemma characterizes the homogenized coefficients $\hat{A}$, the proof can be found in \cite{Jikov1994} sect. 1.6.
\begin{lemma}
	Let $\kappa_1, \kappa_2$ define as previous, then $\hat{A}$ is symmetric and uniformly elliptic, means the following relation
	\[
	\begin{aligned}
	& \hat{A}_{ij} = \hat{A}_{ji} \\
	& \tilde{\kappa}_1 \Seminorm{\xi}^2 \leq \hat{A}_{ij}(\bm{y})\xi_i \xi_j \leq \tilde{\kappa}_2 \Seminorm{\xi}^2 ~~~~ \forall \xi \in \R^d
	\end{aligned}
	\]
	holds, where $\tilde{\kappa}_1, \tilde{\kappa}_2$ depends on $\kappa_1, \kappa_2$.
\end{lemma}
Then we have the homogenized hemivariational form:
\begin{equation}\label{eq:contact problem homohemiform}
\left\{
\begin{aligned}
&\text{Find } u_0 \in V, \text{ s.t. } \forall v \in V \\
& \int_\Omega \hat{A} \nabla u_0 \cdot \nabla v + j^0(\gamma_j u_0; \gamma_j v) \geq \FuncAction{\tilde{f}}{v}
\end{aligned}
\right. .
\end{equation}

Again, \cite{Han2017} tells us $u_0$ exists and is unique. To deduce $u_\epsilon \rightharpoonup u_0$, we first claim that $\{ u_\epsilon\}$ is uniformly bounded, then a subsequence of $\{ u_\epsilon\}$ will weakly converge in $V$. Finally, we prove the convergence point can only be $u_0$.
\begin{lemma}
	There is a constant $C$ independent with $\epsilon$, such that $\Norm{u_\epsilon}_V \leq C$, and 
	\[
	C = 1/\Delta\Brackets{c_0c_j+\Norm{\tilde{f}}_{V^*}} \leq 1/\Delta \Brackets{c_0c_j+\Norm{f}_{V^*}+c_k\Norm{g}_{0,\Gamma_N}}.
	\]
\end{lemma}
\begin{proof}
	Let $v=-u_\epsilon$ in \cref{eq:contact problem hemiform}, we have 
	\[
	\begin{aligned}
	\kappa_1 \Norm{u_\epsilon}_V^2 &\leq \int_{\Omega} A^\epsilon \nabla u_\epsilon\cdot \nabla u_\epsilon \leq \FuncAction{\tilde{f}}{u_\epsilon}+j^0(\gamma_ju_\epsilon; -\gamma_j u_\epsilon) \\
	&\leq \Norm{\tilde{f}}_{V^*} \Norm{u_\epsilon}_V+j^0(\gamma_ju_\epsilon; -\gamma_j u_\epsilon)
	\end{aligned}.	
	\]
	Recall \cref{ass:B2}, then
	\[
	\begin{aligned}
	j^0(\gamma_ju_\epsilon; -\gamma_j u_\epsilon) &\leq j^0(\gamma_ju_\epsilon; -\gamma_ju_\epsilon)+j^0(0;\gamma_ju_\epsilon)-j^0(0;\gamma_ju_\epsilon) \\ &\leq \alpha_j \Norm{\gamma_ju_\epsilon}_{V_j}^2-j^0(0;\gamma_ju_\epsilon) \\
	& \leq \alpha_j c_j^2 \Norm{u_\epsilon}_V^2-j^0(0;\gamma_ju_\epsilon)
	\end{aligned}.
	\]
	By the definition of $j^0(x;h)$, we choose arbitrarily a $\xi \in \partial j(0)$
	\[
	\begin{aligned}
	-j^0(0;\gamma_ju_\epsilon) &\leq -\FuncAction{\xi}{\gamma_ju_\epsilon} \leq \Norm{\xi}_{X_j^*} \Norm{\gamma_ju_\epsilon}_{X_j}\\
	& \leq c_0 c_j \Norm{u_\epsilon}_V
	\end{aligned}.
	\]
	Then we establish the desired inequality.
\end{proof}
div-curl lemma states as follows:
\begin{lemma}[see \cite{Jikov1994} Lem. 1.1]
	Let $\bm{p}_\epsilon, \bm{q}_\epsilon \in L^2(\Omega)^d$, such that:
	\[
	\bm{p}_\epsilon \rightharpoonup \bm{p}_0 ~~~~ \bm{q}_\epsilon \rightharpoonup \bm{q}_0 ~~~~ \text{in } L^2(\Omega)^d .
	\]
	If, in addition $\bm{q}_\epsilon = \nabla v_\epsilon$ and $\Div \bm{p}_\epsilon \rightarrow f$ in $H^{-1}(\Omega)$, then $\bm{p}_\epsilon \cdot \bm{q}_\epsilon \overset{*}{\rightharpoonup} \bm{p}_0 \cdot \bm{q}_0$. Here "$\overset{*}{\rightharpoonup}$" stands for \emph{*-weak} convergence, it means that $\forall \phi \in C^\infty_0(\Omega), \int_\Omega \bm{p}_\epsilon \cdot \bm{q}_\epsilon \phi \rightarrow \int_{\Omega} \bm{p}_0 \cdot \bm{q}_0 \phi$. 
\end{lemma}

Now we can prove the main homogenization result:
\begin{theorem}
Under the assumptions \cref{ass:A}-\cref{ass:C}, let $u_\epsilon$ and $u_0$ be the unique solution of \cref{eq:contact problem hemiform} and \cref{eq:contact problem homohemiform} respectively. Then $u_\epsilon \rightharpoonup u_0$ in $V$.
\end{theorem}
\begin{proof}
	First, we make some justifications. Take $v \in C^\infty_0(\Omega)$ and $-v$ into \cref{eq:contact problem homohemiform}, it is obvious that $j^0(\gamma_ju_\epsilon; \gamma_j v)=j^0(\gamma_j u_\epsilon; -\gamma_j v)=0$. We get $\int_\Omega A^\epsilon \nabla u_\epsilon \cdot \nabla v = \FuncAction{f}{v}$, which means $-\Div(A^\epsilon \nabla u_\epsilon) = f$ in the sense of weak derivative, and $f$ can naturally embed into $H^{-1}(\Omega)$. 
	
	We have already obtained that $\{ u_\epsilon\}$ is uniformly bounded, then up to a subsequence, we have $\nabla u_\epsilon \rightharpoonup \nabla u_0$, and $A^\epsilon \nabla u_\epsilon \rightharpoonup \bm{p}_0$ in $L^2(\Omega)^d$. And then we show $\bm{p}_0=\hat{A}\nabla u_0$ and $u_0$ satisfies \cref{eq:contact problem homohemiform}. Combine div-curl lemma and \cref{thm:oscillating converge} we get
	\[
	A^\epsilon_{ij}\partial_j u_\epsilon \partial_i (\epsilon N^\epsilon_l \lambda_l+\bm{\lambda}\cdot \bm{x}) \overset{*}{\rightharpoonup} \bm{p}_0 \cdot \bm{\lambda} .
	\]
	On the other side, notice $A^\epsilon$ and $\hat{A}$ are both symmetric, from the definition of correctors we have $\partial_j\Brackets{A^\epsilon_{ij} \partial_i (\epsilon N^\epsilon_l \lambda_l + \bm{\lambda}\cdot \bm{x})}=0$, and 
	\[
	A^\epsilon_{ij}\partial_j u_\epsilon \partial_i (\epsilon N^\epsilon_l \lambda_l+\bm{\lambda}\cdot \bm{x}) = A^\epsilon_{ij} \partial_j(\epsilon N^\epsilon_l \lambda_l + \bm{\lambda}\cdot \bm{x})\partial_i u_\epsilon \overset{*}{\rightharpoonup} \hat{A}\bm{\lambda} \cdot \nabla u_0 = \hat{A} \nabla u_0 \cdot \bm{\lambda} .
	\]
	It asserts that $\bm{p}_0=\hat{A}\nabla u_0$. 
	
	Recall that $\gamma_j$ is a compact operator from $V$ to $V_j$, then up to a subsequence $\{\gamma_ju_\epsilon \}$ converges strongly in $V_j$. Because of $u_\epsilon \rightharpoonup u_0$, $\gamma_ju_\epsilon \rightarrow \gamma_ju_0$ in $L^2(\Gamma_C)$ holds. Utilize the fact that $j^0(x; h)$ is upper semicontinuous with $x$ (see \cite{Denkowski2003} Prop. 5.6.6), we arrive at $j^0(\gamma_ju_0; \gamma_j v) \geq \limsup_\epsilon j^0(\gamma_ju_\epsilon; \gamma_j v)$, take the sup limit in the left hand of \cref{eq:contact problem hemiform}, we have $\forall v \in V$:
	\[
	\int_{\Omega} \hat{A} \nabla u_0 \cdot \nabla v +j^0(\gamma_ju_0; \gamma_j v) \geq \FuncAction{\tilde{f}}{v},
	\]
	and this finishes the proof.
\end{proof}

\section{$\epsilon^{1/2}$ estimation}
\label{sec:estimation}
Following \cref{lem:3.1}, \cref{lem:3.2} and \cref{lem:3.4} are quoted from \cite{Shen2016}. In those the standard smoothing operator $\Sop_\epsilon u = \phi_\epsilon \ast u$ is defined as usual (see \cite{Grisvard2011} sect. 7.2), while we need the convolution kernel to be contained in $B_{1/2}(\bm{0}) \subset Q$, $B_r(\bm{x})$ means an open ball centers in $\bm{x}$ with radius $r$.
\begin{lemma}\label{lem:3.1}
	Let $u \in H^1(\R^d)$. Then $\partial_i \Sop_\epsilon(u)=\Sop_\epsilon(\partial_i u)$,
	\[
	\Norm{\Sop_\epsilon u}_{0, \R^d} \leq \Norm{u}_{0, \R^d} ,
	\]
	and there exists a constant $C$ only depends on $d$, such that:
	\[
	\Norm{\Sop_\epsilon u - u}_{0, \R^d} \leq \epsilon C  \Norm{\nabla u}_{0, \R^d} .
	\]
\end{lemma}
\begin{lemma}\label{lem:3.2}
	Let $f \in L^2_{\text{loc}}(\R^d)$ be a 1-periodic function. Then there exists a constant $C$ only depends on $d$, such that for any $u \in L^2(\R^d)$,
	\[
	\Norm{f^\epsilon \Sop_\epsilon u}_{0, \R^d} \leq C\Norm{f}_{0, Q}\Norm{u}_{0,\R^d} .
	\]
\end{lemma}
Denote $\tilde{\Omega}_{\epsilon}=\{\bm{x} \in \R^d: \text{dist}(\bm{x}, \partial \Omega) < \epsilon \}$ as a boundary strip to $\Omega$ with width $2\epsilon$, and $\Omega_{\epsilon}=\{\bm{x} \in \Omega: \text{dist}(\bm{x}, \partial \Omega) < \epsilon \}$. We have an estimation:
\begin{lemma} \label{lem:3.4}
	Let $\Omega$ be a bounded Lipschitz domain in $\R^d$ and $f \in L^2_{\text{loc}}(\R^d)$ a 1-periodic function. Then there exists a constant $C$ only depends on $\Omega$, such that for any $u \in H^1(\R^d)$
	\[
	\int_{\tilde{\Omega}_{\epsilon}} \Seminorm{f^\epsilon}^2 \Seminorm{\Sop_\epsilon(u)}^2 \leq C \epsilon \Norm{f}_{0,Q}^2\Norm{u}^2_{1,\R^d} .
	\]
\end{lemma}
\begin{remark}
	We can obtain a similar result comparing to \cref{lem:3.2} by assuming $\Norm{f}_{L^\infty(\Omega)} < \infty$, but in here the periodic property plays key role and leads to relax on regularity assumption for $f$, which provides us more generality in the estimation.
\end{remark}

Recall the domain we consider has Lipschitz boundary, then the extension operator $\mathbf{E}: H^1(\Omega)\mapsto H^1(\R^d)$ exists and is bounded. Let $w_\epsilon = u_\epsilon - u_0-\epsilon N^\epsilon_l \Sop_\epsilon(\partial_l \overbar{u}_0)$, $\overbar{u}_0 = \mathbf{E}u_0 \in H^1(\R^d)$ is the extension of $u_0$, our goal in this section is to prove following key lemma:
\begin{lemma}\label{lem:key}
	Let \cref{ass:A} be satisfied, and $u_\epsilon, u_0$ the solution of \cref{eq:contact problem hemiform} \cref{eq:contact problem homohemiform} respectively, and suppose that $u_0 \in H^2(\Omega)$, then $\forall v \in V$:
	\[
	\int_{\Omega} A^\epsilon \nabla w_\epsilon \cdot \nabla v \leq j^0(\gamma_ju_0; \gamma_j v)+j^0(\gamma_j u_\epsilon; -\gamma_j v) +  C \Norm{u_0}_{2,\Omega}
	\Brackets{\epsilon^{1/2}\Norm{\nabla v}_{0,\Omega_{2\epsilon}}+\epsilon\Norm{\nabla v}_{0,\Omega}},
	\]
	here constant $C$ dependents on $\Omega, \kappa_1, \kappa_2$ and $w_\epsilon = u_\epsilon - u_0-\epsilon N^\epsilon_l \Sop_\epsilon(\partial_l \overbar{u}_0)$.
\end{lemma}
\begin{proof}
	It is harmless to assume $v \in C^\infty (\Omega) \cap V$. Since $u_0 \in H^2(\Omega)$, we have $\Norm{\overbar{u}_0}_{2, \R^d} \leq C \Norm{u}_{2,\Omega}$. In following proof, constant $C$ only depends on $\Omega, \kappa_1, \kappa_2$. By Green formula,  we have $\forall v \in V$
	\[
	\begin{aligned}
	\int_{\Omega} A^\epsilon \nabla u_\epsilon \cdot \nabla v - \int_{\Gamma_C} \bm{n}\cdot A^\epsilon\nabla u_\epsilon v &= \FuncAction{f}{v}+\int_{\Gamma_N} g v \\
	&= \int_{\Omega} \hat{A} \nabla u_0 \cdot \nabla v - \int_{\Gamma_C}  \bm{n} \cdot \hat{A} \nabla u_0 v 
	\end{aligned}.
	\]
	The definition of $j^0$ gives:
	\[
	\begin{aligned}
	\int_{\Omega} A^\epsilon \nabla u_\epsilon \cdot \nabla v &= \int_{\Omega} \hat{A} \nabla u_0 \cdot \nabla v + \int_{\Gamma_C}  \Brackets{-\bm{n} \cdot \hat{A} \nabla u_0} v + \int_{\Gamma_C} \Brackets{-\bm{n}\cdot A^\epsilon\nabla u_\epsilon}\Brackets{-v} \\
	& \leq \int_{\Omega} \hat{A} \nabla u_0 \cdot \nabla v+j^0(\gamma_j u_0; \gamma_j v)+ j^0(\gamma_j u_\epsilon; -\gamma_j v)
	\end{aligned}.
	\]
	Take a direct calculation,
	\[
	\begin{aligned}
	&\int_\Omega A^\epsilon \nabla w_\epsilon \cdot \nabla v \\
	=& \int_{\Omega} A^\epsilon \nabla u_\epsilon \cdot \nabla v-\int_{\Omega} A^\epsilon \nabla u_0 \cdot \nabla v - \int_{\Omega}A^\epsilon_{ij}\Brackets{\partial_j N_l}^\epsilon \Sop_\epsilon(\partial_l \overbar{u}_0) \partial_i v -\int_{\Omega}\epsilon A^\epsilon_{ij}N^\epsilon_l \Sop_\epsilon(\partial^2_{lj}\overbar{u}_0)\partial_i v \\
	\leq& j^0(\gamma_j u_0; \gamma_j v)+ j^0(\gamma_j u_\epsilon; -\gamma_j v) \\
	&+\int_{\Omega} \Braces{\hat{A}_{ij} \partial_j u_0-A^\epsilon_{ij}\partial_j u_0 - A^\epsilon_{ij}\Brackets{\partial_j N_l}^\epsilon \Sop_\epsilon(\partial_l \overbar{u}_0)}\partial_i v\\
	&-\int_{\Omega}\epsilon A^\epsilon_{ij}N^\epsilon_l \Sop_\epsilon(\partial^2_{lj}\overbar{u}_0)\partial_i v
	\end{aligned}.
	\]
	The techniques for last two parts of above inequality are exactly same as (\cite{Shen2016} Lem. 3.5). However, we will still elaborate its details for the coherence and self-containing of the proof.
	\[
	\begin{aligned}
	&\int_{\Omega} \Braces{\hat{A}_{ij} \partial_j u_0-A^\epsilon_{ij}\partial_j u_0 - A^\epsilon_{ij}\Brackets{\partial_j N_l}^\epsilon \Sop_\epsilon(\partial_l \overbar{u}_0)}\partial_i v-\int_{\Omega}\epsilon A^\epsilon_{ij}N^\epsilon_l \Sop_\epsilon(\partial^2_{lj}\overbar{u}_0)\partial_i v \\
	\coloneqq & \int_\Omega B^\epsilon_{il} \Sop_\epsilon(\partial_l \overbar{u}_0) \partial_i v+\int_\Omega T_\epsilon \cdot \nabla v \\
	\coloneqq & J_1+J_2
	\end{aligned}.
	\]
	Here $B^\epsilon=\SquareBrackets{B_{il}(\bm{x}/\epsilon)}_{1\leq i,l \leq n} \coloneqq \hat{A}-A^\epsilon-\Brackets{A_{ij}\partial_j N_l}^\epsilon$, and 
	\[
	T_\epsilon \coloneqq \Brackets{\hat{A}\nabla u_0 - \hat{A}\Sop_\epsilon(\nabla \overbar{u}_0)}+\Brackets{A^\epsilon\nabla u_0-A^\epsilon \Sop_\epsilon(\nabla \overbar{u}_0)}-\epsilon A^\epsilon_{ij} N^\epsilon_l \Sop_\epsilon(\partial^2_{lj}\overbar{u}_0),
	\]
	\cref{lem:3.1} together with \cref{lem:3.2} implies $\Norm{T_\epsilon}_{0,\Omega} \leq \epsilon \Norm{u_0}_{2,\Omega}$, then:
	\[
	\Seminorm{J_2} \leq C \epsilon \Norm{u_0}_{2,\Omega} \Norm{\nabla v}_{0,\Omega} .
	\]
	Because $\partial_i B_{il}=\bm{0}$ and $\fint_Q B_{il} = \bm{0}$, we can construct $E_{ijl}(\bm{y}) \in H^1_\sharp(Q)$, such that $B_{il}=\partial_j E_{ijl}$, $E_{ijl}=-E_{jil}$, and $\Norm{E_{ijl}}_{1,Q} \leq C \Norm{B_{il}}_{0,Q} \leq C(\kappa_1, \kappa_2)$, see (\cite{Jikov1994} p. 6-7). Take $\theta_\epsilon \in C^\infty(\R^d)$, let $\theta_\epsilon(\bm{x}) \equiv 1$ when $\text{dist}(\bm{x}, \partial \Omega) \leq \epsilon$; $\theta_\epsilon(\bm{x}) \equiv 0$ when $\text{dist}(\bm{x}, \partial \Omega) \geq 2\epsilon$; $0 \leq \theta_\epsilon(\bm{x}) \leq 1$ when $\epsilon \leq \text{dist}(\bm{x}, \partial \Omega) \leq 2\epsilon$. We have $\Norm{\nabla \theta_\epsilon}_{L^\infty(\R^d)} \leq C(\Omega)/\epsilon$. Then split $J_1$ as
	\[
	J_1 = \int_\Omega (\partial_j E_{ijl})^\epsilon \theta_\epsilon \Sop_\epsilon(\partial_l \overbar{u}_0) \partial_i v +\int_\Omega (\partial_j E_{ijl})^\epsilon (1-\theta_\epsilon) \Sop_\epsilon(\partial_l \overbar{u}_0) \partial_i v  \coloneqq J_{11}+J_{12} .
	\]
	Use \cref{lem:3.4}, we have:
	\[
	\begin{aligned}
	\Seminorm{J_{11}} \leq &\Norm{\nabla v}_{0,\Omega_{2\epsilon}} \SquareBrackets{\sum_i\int_{\Omega_{2\epsilon}} \Seminorm{(\partial_j E_{ijl})^\epsilon}^2 \Seminorm{\Sop_\epsilon(\partial_l \overbar{u}_0)}^2 \theta_\epsilon^2}^{1/2} \\
	\leq & \Norm{\nabla v}_{0,\Omega_{2\epsilon}} \SquareBrackets{\sum_i\int_{\tilde{\Omega}_{2\epsilon}} \Seminorm{(\partial_j E_{ijl})^\epsilon}^2 \Seminorm{\Sop_\epsilon(\partial_l \overbar{u}_0)}^2}^{1/2} \\
	\leq  & C \epsilon^{1/2} \Norm{u_0}_{2,\Omega}\Norm{\nabla v}_{0,\Omega_{2\epsilon}}
	\end{aligned}.
	\]
	Due to $1-\theta_\epsilon \equiv 0$ on $\partial \Omega$, $(1-\theta_\epsilon) \Sop_\epsilon(\partial_l \overbar{u}_0) \partial_i v \in C^\infty_0(\Omega)$, and $(\partial_j E_{ijl})^\epsilon=\epsilon \partial_j E_{ijl}^\epsilon$. According to integration by parts:
	\[
	\begin{aligned}
	J_{12}=&\int_{\Omega} \epsilon \partial_j E_{ijl}^\epsilon (1-\theta_\epsilon) \Sop_\epsilon(\partial_l \overbar{u}_0) \partial_i v 
	= -\int_{\Omega} \epsilon E_{ijl}^\epsilon \partial_j\SquareBrackets{(1-\theta_\epsilon)\Sop_\epsilon(\partial_l \overbar{u}_0)\partial_i v} \\
	= & \int_{\Omega} E_{ijl}^\epsilon \epsilon \partial_j \theta_\epsilon \Sop_\epsilon(\partial_l \overbar{u}_0) \partial_i v \\
	& - \int_{\Omega} \epsilon E_{ijl}^\epsilon (1-\theta_\epsilon) \Sop_\epsilon(\partial^2_{jl}\overbar{u}_0) \partial_i v \\
	& - \int_{\Omega} \epsilon E_{ijl}^\epsilon (1-\theta_\epsilon) \Sop_\epsilon(\partial_l \overbar{u}_0) \partial^2_{ij} v	
	\end{aligned},
	\]
	$E_{ijl}=-E_{jil}$ implies $\int_{\Omega} \epsilon E_{ijl}^\epsilon (1-\theta_\epsilon) \Sop_\epsilon(\partial_l \overbar{u}_0) \partial^2_{ij} v =0$, and
	\[
	\begin{aligned}
	\int_{\Omega} \Seminorm{E_{ijl}^\epsilon \epsilon \partial_j \theta_\epsilon \Sop_\epsilon(\partial_l \overbar{u}_0) \partial_i v} \leq& C \Norm{\nabla v}_{0,\Omega_{2\epsilon}} \SquareBrackets{\sum_{i,j}\int_{\tilde{\Omega}_{2\epsilon}} \Seminorm{E^\epsilon_{ijl}}^2 \Seminorm{\Sop_\epsilon(\partial_l \overbar{u}_0)}^2}^{1/2} \\\leq& C \epsilon^{1/2} \Norm{\nabla v}_{0,\Omega_{2\epsilon}} \Norm{u_0}_{2,\Omega}
	\end{aligned},
	\]
	\[
	\begin{aligned}
	\int_{\Omega} \Seminorm{\epsilon E_{ijl}^\epsilon (1-\theta_\epsilon)\Sop_\epsilon(\partial_{jl}\overbar{u}_0)\partial_i v} \leq& \epsilon \Norm{\nabla v}_{0,\Omega} \SquareBrackets{\sum_i \int_{\Omega}  \Seminorm{E_{ijl}^\epsilon}^2 \Seminorm{\Sop_\epsilon(\partial_{jl}\overbar{u}_0)}^2}^{1/2} \\ \leq& C \epsilon \Norm{\nabla v}_{0,\Omega} \Norm{u_0}_{2,\Omega}
	\end{aligned}.	
	\]
	Finally, we derive $\Seminorm{J_1}+\Seminorm{J_2}\leq C \Norm{u_0}_{2,\Omega}
	\Brackets{\epsilon^{1/2}\Norm{\nabla v}_{0,\Omega_{2\epsilon}}+\epsilon\Norm{\nabla v}_{0,\Omega}}$. 
\end{proof}
A corollary of this lemma is following theorem:
\begin{theorem}
	Let the assumptions \cref{ass:A}-\cref{ass:C} be satisfied, and $u_0 \in H^2(\Omega)$. Then
	\[
	\Norm{\nabla w_\epsilon}_{0,\Omega} \leq C\epsilon^{1/2}\Norm{u_0}_{2,\Omega},
	\]
	here constant $C$ depends on $\Omega, \kappa_1, \kappa_2, \Delta$.
\end{theorem}
\begin{proof}
	Take $\theta_\epsilon$ defined as previous, and let $\tilde{u}_\epsilon=u_0+\epsilon N_l^\epsilon \Sop_\epsilon(\partial_l \overbar{u}_0)(1-\theta_\epsilon)$. Then $w_\epsilon = u_\epsilon-\tilde{u}_\epsilon-\epsilon N_l^\epsilon \Sop_\epsilon(\partial_l \overbar{u}_0)\theta_\epsilon \coloneqq u_\epsilon-\tilde{u}_\epsilon-r_\epsilon$. It is obvious to see $\tilde{u}_\epsilon \in V$ and $\gamma_j u_0 = \gamma_j \tilde{u_\epsilon}$. We use \cref{lem:3.1}-\cref{lem:3.4} to handle $r_\epsilon = \epsilon N^\epsilon_l \Sop_\epsilon(\partial_l \overbar{u}_0) \theta_\epsilon$:
	\[
	\begin{aligned}
	\Norm{\partial_i r_\epsilon}_{0,\Omega} \leq& \Norm{(\partial_i N_l)^\epsilon \Sop_\epsilon(\partial_l \overbar{u}_0)\theta_\epsilon}_{0,\Omega} + \Norm{N_l^\epsilon \epsilon \partial_i \theta_\epsilon \Sop_\epsilon(\partial_l \overbar{u}_0)}_{0,\Omega}\\
	&+\Norm{\epsilon N_l \theta_\epsilon \Sop_\epsilon(\partial_{il}^2 \overbar{u}_0)}_{0,\Omega} \\
	\leq& C(\Omega, \kappa_1, \kappa_2)\Brackets{\epsilon^{1/2}\Norm{u_0}_{2,\Omega}+\epsilon \Norm{u_0}_{2,\Omega}} \leq C(\Omega, \kappa_1, \kappa_2) \epsilon^{1/2} \Norm{u_0}_{2,\Omega}
	\end{aligned}.
	\]
	Substitute $v=u_\epsilon-\tilde{u}_\epsilon \in V$ into \cref{lem:key}, and recall $\gamma_j u_0 = \gamma_j \tilde{u_\epsilon}$, we have:
	\[
	\begin{aligned}
	&\int_{\Omega} A^\epsilon \nabla(u_\epsilon - \tilde{u}_\epsilon) \cdot \nabla (u_\epsilon -\tilde{u}_\epsilon) \\
	=& \int_\Omega A^\epsilon \nabla w_\epsilon \cdot \nabla (u_\epsilon -\tilde{u}_\epsilon) + \int_{\Omega} A^\epsilon \nabla r_\epsilon \cdot \nabla (u_\epsilon -\tilde{u}_\epsilon) \\
	\leq& j^0(\gamma_j u_0; \gamma_j (u_\epsilon-\tilde{u}_\epsilon))+j^0(\gamma_j u_\epsilon; \gamma_j(\tilde{u}_\epsilon-u_\epsilon))\\
	&+C\epsilon^{1/2}\Norm{u_0}_{2,\Omega} \Norm{\nabla(u_\epsilon -\tilde{u}_\epsilon)}_{0, \Omega}+\kappa_2 \Norm{\nabla r_\epsilon}_{0,\Omega} \Norm{\nabla (u_\epsilon-\tilde{u}_\epsilon)}_{0,\Omega} \\		
	\leq& j^0(\gamma_j \tilde{u}_\epsilon; \gamma_j (u_\epsilon-\tilde{u}_\epsilon))+j^0(\gamma_j u_\epsilon; \gamma_j(\tilde{u}_\epsilon-u_\epsilon))+C\epsilon^{1/2}\Norm{u_0}_{2,\Omega} \Norm{\nabla(u_\epsilon -\tilde{u}_\epsilon)}_{0,\Omega}\\
	\leq& \alpha_j c_j^2 \Norm{u_\epsilon-\tilde{u}_\epsilon}_V^2 + C\epsilon^{1/2} \Norm{u_0}_{2,\Omega} \Norm{u_\epsilon-\tilde{u}_\epsilon}_V
	\end{aligned}.
	\]
	By a direct calculation, we have $\Norm{u_\epsilon-\tilde{u}_{\epsilon}}_V \leq 1/\Delta C\epsilon^{1/2} \Norm{u_0}_{2,\Omega}$. Then $\Norm{\nabla w_\epsilon}_{0,\Omega} \leq C\epsilon^{1/2} \Norm{u_0}_{2,\Omega}$, $C$ depends on $\Omega, \kappa_1, \kappa_2, \Delta$, since $\Norm{\nabla w_\epsilon}_{0,\Omega} \leq \Norm{u_\epsilon-\tilde{u}_\epsilon}_V+ \Norm{\nabla r_\epsilon}_{0,\Omega}$.
\end{proof}
We have following corollary to quantify the $L^2$ convergence rate:
\begin{corollary}
	Let the assumptions \cref{ass:A}-\cref{ass:C} be satisfied, and $u_0 \in H^2(\Omega)$. Then:
	\[
	\Norm{u_\epsilon - u_0}_{0,\Omega} \leq C\epsilon^{1/2} \Norm{u_0}_{2,\Omega}.
	\]
	Here the constant $C$ depends on $\Omega,\kappa_1,\kappa_2,\Delta$
\end{corollary}
\begin{proof}
	Since $\Gamma_D \neq \varnothing$, by Poincar\'{e} inequality, $\Norm{u_\epsilon-\tilde{u}_\epsilon}_{0,\Omega} \leq C(\Omega) \Norm{u_\epsilon-\tilde{u}_\epsilon}_V \leq C\epsilon^{1/2} \Norm{u_0}_{2,\Omega}$. It is easy to show $\Norm{\tilde{u}_\epsilon - u_0}_{0,\Omega}\leq C(\Omega) \epsilon \Norm{u_0}_{1,\Omega}$. Then the triangle relation tells us $\Norm{u_\epsilon - u_0}_{0,\Omega} \leq C\epsilon^{1/2}\Norm{u_0}_{2,\Omega}$.
\end{proof}
\begin{remark}
	This $L^2$ convergence rate presented here is not optimal while We conjecture that best results is $\Norm{u_\epsilon-u_0}_{0,\Omega} \leq C\epsilon\Norm{u_0}_{2,\Omega}$ (see \cite{Shen2016}), and \cref{sec:experiments} supports this claim. However, gradient information is far more valuable in application, this is why we mainly consider norm $\Seminorm{\cdot}_{1,\Omega}$ or $\Norm{\nabla \cdot}_{0,\Omega}$.
\end{remark}

\section{A Special Case---Robin Problem}
\label{sec:robin}

The assumptions \cref{ass:B1}-\cref{ass:C} are relatively general to cover considerable situations. For example, let us consider a simplified version of \cref{eq:contact problem hemiform}:
\begin{proposition} \label{pp:simplified contact problem}
Let $\mathrm{B}(x)$ be a function on $\R$ satisfying uniform Lipschitz condition, that is $\Seminorm{\mathrm{B}(x)-\mathrm{B}(y)} \leq L_\mathrm{B} \Seminorm{x-y} \forall x, y \in \R$, and the other notations are defined as previous. Then the following hemivariational inequality (variational equality) is solvable:
\begin{equation}\label{eq:simplified contact problem}
\left\{
	\begin{aligned}
		&\text{Find } u \in V, \text{ s.t. } \forall v \in V \\
		&\int_\Omega A \nabla u_\epsilon \cdot \nabla v + \int_{\Gamma_C} \mathrm{B}\Brackets{u} v = \FuncAction{f}{v}+\int_{\Gamma_N} gv 
	\end{aligned} ~~,	
\right.
\end{equation}
if $\kappa_1-L_\mathrm{B} c_j^2 > 0$.
\end{proposition}

We can reprove this proposition by directly utilizing strongly monotone operator theory (see \cite{Zeidler1995} sect. 2.14). Moreover, in a special case--- Robin problems, we can obtain $L^2(\Omega)$ estimation by utilizing the bilinearity of its variational form.

Let $A^\epsilon$ define as previous and satisfy the assumption \cref{ass:A}, for simplicity, we state Robin problem on the whole boundary $\Gamma$:
\begin{equation}
\left\{
\begin{aligned}
-\Div(A^\epsilon \nabla u_\epsilon) = f ~~~~ &\text{ in } L^2(\Omega) \\
\bm{n}\cdot A^\epsilon \nabla u_\epsilon + \alpha(\bm{x})u_\epsilon = g ~~~~ &\text{ in } L^2(\Gamma)
\end{aligned} .
\right.
\end{equation}
And the corresponding variational form is:
\begin{equation} \label{eq:robin variation}
\left\{
\begin{aligned}
&\text{Find } u_\epsilon \in H^1(\Omega), \text{ s.t. } \forall v \in H^1(\Omega) \\
&\int_{\Omega} A^\epsilon \nabla u_\epsilon \cdot \nabla v + \int_{\Gamma} \alpha u_\epsilon v = \int_{\Omega} fv + \int_{\Gamma} gv
\end{aligned} .
\right.
\end{equation}
We need assumption 
\[
0 < \alpha_1 \leq \alpha(\bm{x}) \leq \alpha_2 < \infty ~~ \forall \bm{x} \in \Gamma 
\] 
to prove this bilinear form is coercive, then Lax-Milgram theorem asserts the solvability of  \cref{eq:robin variation}. We have following lemma, the proof is postponed in appendix.
\begin{lemma}\label{lem:coercive for Robin}
	Let $\Omega$ be a Lipschitz domain and $\Gamma$ its boundary. Then there exists a constant $C(\Omega)$, such that $\forall \psi \in H^1(\Omega)$:
	\[
	C(\Omega) \Norm{\psi}_{1,\Omega}^2 \leq \int_{\Omega} \Seminorm{\nabla \psi}^2 + \int_{\Gamma} \psi^2 .
	\]
\end{lemma}
Similarly, we can prove 
\begin{theorem}
	The solutions $u_\epsilon$ of Robin problems \cref{eq:robin variation} converge weakly to $u^0$, while $u^0$ is the solution of homogenized Robin problem:
	\begin{equation}
	\left\{
	\begin{aligned}
	-\Div(\hat{A}\nabla u_0) = f ~~~~ &\text{ in } L^2(\Omega) \\
	\bm{n}\cdot \hat{A}\nabla u_0+\alpha(\bm{x})u_0=g ~~~~ &\text{ in } L^2(\Gamma)
	\end{aligned} .
	\right.
	\end{equation}
	And the associated variation form is:
	\begin{equation}\label{eq:homogenized robin problem variation form}
	\left\{
	\begin{aligned}
	&\text{Find} ~~ u_0 \in H^1(\Omega) \text{ s.t. } \forall v \in H^1(\Omega)\\
	&\int_{\Omega}\hat{A} \nabla u_0 \cdot \nabla v + \int_{\Gamma} \alpha u_0 v = \int_{\Omega}fv + \int_{\Gamma} gv
	\end{aligned} .
	\right.
	\end{equation}
\end{theorem}

Estimation \cref{lem:boundary estimation} is cited from Thm 1.5.1.10 in \cite{Grisvard2011}, we use this lemma to prove $O(\epsilon^{1/2})$ convergence rate in $H^1(\Omega)$ norm.
\begin{lemma}\label{lem:boundary estimation}
	Let $\Omega$ be a Lipschitz domain in $\R^d$, then for $u \in H^1(\Omega)$,
	\[
	\int_{\Gamma} u^2 \leq C(\Omega) \Brackets{ t\int_{\Omega} \Seminorm{\nabla u}^2 + t^{-1} \int_{\Omega} u^2} .
	\]
	Here $t$ can choose arbitrarily in $(0, 1)$.
\end{lemma}
Let $w_\epsilon$ be defined as previous, we have a parallel version of \cref{lem:key}:
\begin{lemma} \label{lem:Robin problem estimation 1/2}
	Assume $u_0 \in H^2(\Omega)$, then $\forall v \in H^1(\Omega)$
	\[
	\int_\Omega A^\epsilon \nabla w_\epsilon \cdot \nabla v + \int_{\Gamma} \alpha w_\epsilon v \leq C(\Omega, \kappa_1, \kappa_2, \alpha_1, \alpha_2) \Norm{u_0}_{2,\Omega} \Brackets{\epsilon^{1/2}\Norm{\nabla v}_{0,\Omega_{2\epsilon}}+\epsilon\Norm{\nabla v}_{0,\Omega}+\epsilon \Norm{v}_{0, \Gamma}} .
	\]
\end{lemma}
\begin{proof}
	Compare with the proof of \cref{lem:key}, we are left to show:
	\[
	J_3 \coloneqq \int_{\Gamma} \epsilon \alpha N^\epsilon_l \Sop_\epsilon(\partial_l \overbar{u}_0) v \leq C \epsilon \Norm{u_0}_{2,\Omega} \Norm{v}_{0,\Gamma}  .
	\]
	By calculation:
	\[
	\Seminorm{J_3} \leq C\int_{\Gamma} \Seminorm{\epsilon N^\epsilon_l \Sop_\epsilon(\partial_l \overbar{u}_0)\theta_\epsilon v} \leq C\Norm{\epsilon N^\epsilon_l \Sop_\epsilon(\partial_l \overbar{u}_0)}_{0,\Gamma} \Norm{v}_{0,\Gamma} .
	\]
	Then
	\[
	\begin{aligned}
	& \int_{\Gamma} \Seminorm{\epsilon N^\epsilon_l \Sop_\epsilon(\partial_l \overbar{u}_0)\theta_\epsilon}^2\\  
	\underset{\text{use }  \eqref{lem:boundary estimation}}{\leq}&   Ct \sum_i \int_\Omega \Seminorm{(\partial_i N_l)^\epsilon \Sop_\epsilon(\partial_l \overbar{u}_0)\theta_\epsilon}^2+\Seminorm{N^\epsilon_l \Sop_\epsilon(\partial_l \overbar{u}_0)\epsilon \partial_i \theta_\epsilon}^2+\epsilon^2 \Seminorm{N^\epsilon_l \Sop_\epsilon(\partial^2_{il}\overbar{u}_0)\theta_\epsilon}^2 \\
	&+ t^{-1}\epsilon^2 \int_{\Omega} \Seminorm{N^\epsilon_l \Sop_\epsilon(\partial_l \overbar{u}_0)\theta_\epsilon}^2\\
	\underset{\text{take $t=\epsilon$}}{\leq} &C\epsilon \sum_i \int_{\tilde{\Omega}_{2\epsilon}}\Seminorm{(\partial_i N_l)^\epsilon \Sop_\epsilon(\partial_l \overbar{u}_0)}^2+\Seminorm{N^\epsilon_l \Sop_\epsilon(\partial_l \overbar{u}_0)}^2+\epsilon^2\int_{\Omega}\Seminorm{N^\epsilon_l \Sop_\epsilon(\partial^2_{il}\overbar{u}_0)}^2 \\
	&+ \epsilon \int_{\tilde{\Omega}_{2\epsilon}} \Seminorm{N^\epsilon_l \Sop_\epsilon(\partial_l \overbar{u}_0)}^2 \\
	\underset{\text{use}\eqref{lem:3.4}\eqref{lem:3.2}}{\leq} &C\epsilon^2 \Norm{u}_{2,\Omega}^2+\epsilon^3 \Norm{u_0}_{2,\Omega}^2 \leq C\epsilon^2 \Norm{u_0}^2_{2,\Omega}
	\end{aligned} .
	\]
\end{proof}
Then, we have:
\begin{theorem} \label{thm:Robin weakly conv}
	Let $w_\epsilon = u_\epsilon - u_0-\epsilon N^\epsilon_l \Sop_\epsilon(\partial_l \overbar{u}_0)$, $u_\epsilon$ and $u_0$ be the solution of \cref{eq:robin variation} and \cref{eq:homogenized robin problem variation form} respectively. Assume $u_0(\Omega) \in H^2(\Omega)$, then
	\[
	\int_{\Omega} \Seminorm{\nabla w_\epsilon}^2 + \int_{\Gamma} \alpha w_\epsilon^2 \leq C(\Omega, \kappa_1, \kappa_2, \alpha_1, \alpha_2) \epsilon \Norm{u_0}_{2,\Omega}^2 .
	\]
	
	As a corollary, we have
	\[
	\begin{aligned}
	\Norm{\nabla w_\epsilon}_{0,\Omega} &\leq C \epsilon^{1/2} \Norm{u_0}_{2,\Omega} , \\
	\Norm{u_\epsilon - u_0}_{0,\Omega} &\leq C \epsilon^{1/2} \Norm{u_0}_{2,\Omega} .
	\end{aligned}	
	\]
\end{theorem}

Next, we will show that $\Norm{u_\epsilon - u_0}_{0,\Omega}$ can reach $O(\epsilon)$. We need a regularity result for Robin boundary problem:

\begin{proposition}
	Suppose that $\Omega$ has $C^{1,1}$ boundary. In addition to uniformly ellipticity, coefficients $A(x)=\SquareBrackets{A_{ij}}_{1\leq i,j \leq d}$ are in $C^{0,1}(\Omega)$, and $\alpha(\bm{x})$ is $C^{0,1}(\Gamma)$ (in the sense of local coordinate). Then $\forall f \in L^2(\Omega)$, $u$ is the solution of Robin problem:
	\begin{equation}
	\left\{
	\begin{aligned}
	-\Div(A \nabla u) = f  \\
	\bm{n}\cdot A \nabla u + \alpha(\bm{x})u = 0
	\end{aligned} .
	\right.
	\end{equation}
	Then $u \in H^2(\Omega)$ with estimation $\Norm{u}_{2,\Omega} \leq C \Norm{f}_{0,\Omega}$, here $C$ depends on $\Omega, \Norm{A_{ij}}_{C^{0,1}(\Omega)}, \Norm{\alpha}_{C^{0,1}(\Gamma)}$ and $\kappa_1, \kappa_2$.	
\end{proposition}

A proof is provided in the appendix. The $L^2$ estimation states as following:
\begin{theorem}
	Suppose that $\Omega$ has $C^{1,1}$ boundary. In addition to the hypotheses in \cref{thm:Robin weakly conv}, $\alpha(\bm{x})$ is uniformly Lipschitz continuous on $\Gamma$. Then we have:
	\[
	\Norm{u_\epsilon - u_0}_{0,\Omega} \leq C \epsilon \Norm{u_0}_{2,\Omega}. 
	\]
	Where $C=C(\Omega, \kappa_1, \kappa_2, \Norm{\alpha}_{C^{0,1}(\Gamma)})$.
\end{theorem}
\begin{proof}
	It is sufficient to show $\Norm{w_\epsilon}_{0,\Omega} \leq C\epsilon \Norm{u_0}_{2,\Omega}$, because of the fact that $\Norm{\epsilon N_l \Sop_\epsilon (\partial_l \overbar{u}_0)}_{0,\Omega} \leq C\epsilon \Norm{u_0}_{2,\Omega}$.
	
	We will take duality technique from \cite{Shen2016} for the rest proof. First $\forall G \in L^2(\Omega)$, we have $\rho \in H^2(\Omega)$ which satisfies homogenized Robin problem:
	\[
	\left\{
	\begin{aligned}
	-\Div(\hat{A} \nabla \rho_0) = G  \\
	\bm{n}\cdot \hat{A} \nabla \rho_0 + \alpha(\bm{x})\rho_0 = 0
	\end{aligned} .
	\right.
	\]
	We also let $\rho_\epsilon$ be the solution of original Robin problem:
	\[
	\left\{
	\begin{aligned}
	-\Div(A^\epsilon \nabla \rho_\epsilon) = G  \\
	\bm{n}\cdot A^\epsilon \nabla \rho_\epsilon + \alpha(\bm{x})\rho_\epsilon = 0
	\end{aligned} .
	\right.
	\]
	According to the regularity result, we have $\Norm{\rho}_{2,\Omega} \leq C \Norm{G}_{0,\Omega}$. Take $w_\epsilon$ as test function into previous equation,
	\[
	\int_{\Omega} G w_\epsilon = \int_{\Omega} A^\epsilon \nabla \rho_\epsilon \cdot \nabla w_\epsilon + \int_\Gamma \alpha \rho_\epsilon w_\epsilon .
	\]
	Split $\rho_\epsilon$ into three parts $\rho_\epsilon = \rho + \epsilon N_l \Sop_\epsilon(\partial_l \overbar{\rho}) + \Brackets{\rho_\epsilon - \rho - \epsilon N_l \Sop_\epsilon(\partial_l \overbar{\rho})} \eqqcolon \rho + \psi_\epsilon + \eta_\epsilon$. From \cref{thm:Robin weakly conv}, we will have $\Norm{\nabla \eta_\epsilon}_{0,\Omega} \leq \epsilon^{1/2} \Norm{\rho}_{2,\Omega}$ and $\Norm{\eta_\epsilon}_{0,\Gamma} \leq \epsilon^{1/2} \Norm{\rho}_{2,\Omega}$. Then,
	\[
	\begin{aligned}
	\int_{\Omega} G w_\epsilon =& \int_\Omega A^\epsilon \nabla \rho \cdot \nabla w_\epsilon + \int_\Gamma \alpha \rho w_\epsilon \\
	&+ \int_\Omega A^\epsilon \nabla \psi_\epsilon \cdot \nabla w_\epsilon + \int_\Gamma \alpha \psi_\epsilon w_\epsilon \\
	&+ \int_\Omega A^\epsilon \nabla \eta_\epsilon \cdot \nabla w_\epsilon + \int_\Gamma \alpha \eta_\epsilon w_\epsilon \\
	\coloneqq & J_1+J_2+J_3 .
	\end{aligned}
	\]
	For $J_3$, we have $\Seminorm{J_3} \leq C\epsilon \Norm{\rho}_{2,\Omega}\Norm{u_0}_{2,\Omega}$. Use \cref{lem:Robin problem estimation 1/2}, we obtain:
	\[
	\begin{aligned}
		\Seminorm{J_1} &\leq C\Norm{u_0}_{2,\Omega} \Brackets{\epsilon^{1/2}\Norm{\nabla \rho}_{0, \Omega_{2\epsilon}} + \epsilon \Norm{\nabla \rho}_{0,\Omega} + \epsilon \Norm{\rho}_{0,\Gamma}} \\
		&\leq C\Norm{u_0}_{2,\Omega} \Brackets{\epsilon\Norm{\rho}_{2, \Omega} + \epsilon \Norm{\nabla \rho}_{0,\Omega} + \epsilon \Norm{\rho}_{0,\Gamma}} \leq C\epsilon \Norm{\rho}_{2,\Omega} \Norm{u_0}_{2,\Omega}
	\end{aligned}.
	\]
	Similarly, $\Seminorm{J_2} \leq C\epsilon \Norm{\rho}_{2,\Omega} \Norm{u_0}_{2,\Omega}$. Together, we have 
	\[
	\int_\Omega G w_\epsilon \leq C \epsilon \Norm{\rho}_{2,\Omega} \Norm{u_0}_{2,\Omega} \leq C \epsilon \Norm{G}_{0,\Omega} \Norm{u_0}_{2,\Omega}
	\]. 	 
\end{proof}

\section{Computational Method}
\label{sec:computation}
After the completing of $O(\epsilon^{1/2})$ estimation, \cref{eq:contact problem hemiform} will be computable because  $u_0+\epsilon N^\epsilon_l\Sop_\epsilon(\partial_l \overbar{u}_0)$ can approximate well to original high oscillating $u_\epsilon$. However, obtain $\overbar{u}_0$ and perform smoothing action $\Sop_\epsilon$ is impractical in real computation. Instead, we should calculate $\partial_i u_0 + (\partial_i N_l)^\epsilon \partial_l u_0$ as an approximation for $\partial_i u_\epsilon$. Here is a lemma for the error analysis.
\begin{lemma} \label{lem:gradient error analysis}
	Let $u_\epsilon$ and $u_0$ be the solution of \cref{eq:contact problem hemiform} and \cref{eq:contact problem homohemiform} respectively, and assumptions \cref{ass:A}-\cref{ass:C} be satisfied, and assume $u_0 \in H^2(\Omega), N_l \in W^{1,\infty}_\sharp(Q)$. Then:
	\[
	\sum_i \int_{\Omega}\Seminorm{\partial_i u_\epsilon-\partial_i u_0 -(\partial_i N_l)^\epsilon \partial_l u_0}^2 \leq \epsilon C(\Omega, \kappa_1, \kappa_2, \Delta, \Norm{N_l}_{W^{1,\infty}(Q)})  \Norm{u_0}_{2,\Omega}^2 .
	\] 
\end{lemma}
\begin{proof}
	Directly calculate the error $\partial_i w_\epsilon = \partial_i u_\epsilon - \partial_i u_0 - (\partial_i N_l)^\epsilon \Sop_\epsilon(\partial_l \overbar{u}_0) - \epsilon N_l^\epsilon \Sop_\epsilon(\partial^2_{il}\overbar{u}_0)$, \cref{lem:3.2} tells us $\Norm{N_l^\epsilon \Sop_\epsilon(\partial^2_{il}\overbar{u}_0)}_{0,\Omega} \leq C \Norm{u_0}_{2, \Omega}$. By H\"{o}lder inequality, we have
	\[
	\begin{aligned}
	\int_{\Omega} \Seminorm{(\partial_i N_l)^\epsilon}^2 \Seminorm{\Sop_\epsilon (\partial_l \overbar{u}_0)-\partial_l u_0}^2 \leq& C \sum_l \int_{\R^d} \Seminorm{\Sop_\epsilon (\partial_l \overbar{u}_0)-\partial_l \overbar{u}_0}^2 \leq C\epsilon^2 \int_{\R^d} \Seminorm{\nabla^2 \overbar{u}_0}^2 \\
	\leq& C \epsilon^2 \Norm{u_0}_{2,\Omega}^2
	\end{aligned}.
	\]
	Then the conclusion holds because $\Norm{\nabla w_\epsilon}_{0,\Omega}$ dominates the error:
	\[
	\begin{aligned}
	&\sum_i \int_{\Omega} \Seminorm{\partial_i u_\epsilon-\partial_i u_0 -(\partial_i N_l)^\epsilon \partial_l u_0}^2 \\
	\leq & \int_{\Omega} \Seminorm{\nabla w_\epsilon}^2\\
	&+\sum_i\epsilon^2 \int_{\Omega} \Seminorm{N_l^\epsilon \Sop_\epsilon(\partial^2_{il}\overbar{u}_0)}^2 + \sum_i \int_{\Omega} \Seminorm{(\partial_i N_l)^\epsilon}^2 \Seminorm{\Sop_\epsilon (\partial_l \overbar{u}_0)-\partial_l u_0}^2
	\end{aligned}.	
	\]	
\end{proof}
\begin{remark}
	It seems that we can not weaken the regularity assumption for $N_l(\bm{y})$ because we can not prove a strengthened version of  \cref{lem:3.2}, that is:
	\begin{quotation}
		Let $f \in L_{\text{loc}}^2(\R^d)$ be a 1-periodic function, Then for any $u \in H^1(\R^d)$, 
		\[
		\Norm{f^\epsilon \Brackets{\Sop_\epsilon u - u}}_{0,\R^d} \leq C \epsilon \Norm{f}_{0, Q} \Norm{\nabla u}_{0,\R^d}.
		\]
	\end{quotation}
	We also mention that when the coefficients $A(\bm{y})$ is piecewise smooth, which is a suitable assumption in application, and the $W^{1,\infty}$ proposition can be verified by the works in \cite{Li2003}.
\end{remark}

We can implement finite element method (FEM) to obtain the numerical solution of $u_0$. Let $V_h$ be the finite element space (see \cite{Brenner2008}), then
\begin{equation}\label{eq:contact problem homofem}
\left\{
\begin{aligned}
&\text{Find } u_{0,h} \in V_h, \text{ s.t. } \forall v_h \in V \\
& \int_\Omega \hat{A} \nabla u_{0,h} \cdot \nabla v_h + j^0(\gamma_j u_{0,h}; \gamma_j v_h) \geq \FuncAction{\tilde{f}}{v_h}
\end{aligned} .
\right.
\end{equation}
The existence and uniqueness of this problem were also shown by utilizing the framework in \cite{Han2017}. Then the computational method is direct:
\begin{algorithm}[H]
	\caption{Computation framework for contact problem in small periodic setting}
	\begin{algorithmic}[1]
		\State Solve the equations on correctors \cref{eq:correctors} and obtain numerical solution $N_l^\star(\bm{y})$. Calculate numerical homogenized coefficients $\hat{A}^\star$. In this step, the cost of computation is independent with original problem, thus we can implement high accuracy method.
		
		\State Choose grid size $h$ to mesh the domain $\Omega$. Solve \cref{eq:contact problem homofem} and obtain $u_{0,h}$. 
		
		\State Construct numerical gradient value $(\partial_i u_\epsilon)^\star$ by 
		\[
		(\partial_i u_\epsilon)^\star(\bm{x}) \coloneqq \partial_i u_{0,h}(\bm{x}) + \partial_i N_l^\star(\bm{x}/\epsilon) \partial_l u_{0,h}(\bm{x}) .
		\]
	\end{algorithmic}
\end{algorithm}
Rewrite the error analysis \cref{lem:gradient error analysis} in FEM framework, we derive following numerical error expression:
\[
\Norm{\nabla u_\epsilon-(\nabla u_\epsilon)^\star}_{0,\Omega} \leq C\Brackets{\epsilon^{1/2} \Norm{u_0}_{2,\Omega}+\Seminorm{u_0-u_{0,h}}_{1,\Omega}},
\]
here $(\nabla u_\epsilon)^\star=\SquareBrackets{(\partial_i u_\epsilon)^\star}_{1\leq i \leq d}=\SquareBrackets{\partial_i u_{0,h}(\bm{x}) + \partial_i N_l^\star(\bm{x}/\epsilon) \partial_l u_{0,h}(\bm{x})}_{1\leq i \leq d}$. The estimation for $\Seminorm{u_0-u_{0,h}}_{1,\Omega}$ merely involves the theory of FEM. Fortunately, a C\'{e}a's inequality has been proved in \cite{Han2017} sect. 4.2, and combine their works, we have following theorem:
\begin{theorem}
	Assume $u_0(\bm{x}) \in H^2(\Omega)$ and $N_l(\bm{y}) \in W^{1,\infty}_\sharp(Q)$. Neglect the error brought by calculating numerical correctors $N^\star_l(\bm{y})$ and homogenized coefficients $\hat{A}^\star$. Use Lagrange FEM to Solve \cref{eq:contact problem homofem} with $h$ grid size. Recall $\gamma_ju_0 \in H^{3/2}(\Gamma_C)$ and $\Norm{\gamma_ju_0}_{3/2,\Gamma_C} \leq C \Norm{u_0}_{2,\Omega}$. Then error between numerical gradient $(\nabla u_\epsilon)^\star$ and $\nabla u_\epsilon$ satisfies the relation:
	\[
	\Norm{\nabla u_\epsilon-(\nabla u_\epsilon)^\star}_{0,\Omega} \leq C\SquareBrackets{(\epsilon^{1/2}+h) \Norm{u_0}_{2,\Omega}+h^{3/4} \sqrt{\Norm{u_0}_{2,\Omega}}} .
	\]
	If $\gamma_ju_0 \in H^2(\Gamma_C)$, there exists optimal numerical error order:
	\[
	\Norm{\nabla u_\epsilon-(\nabla u_\epsilon)^\star}_{0,\Omega} \leq C (h+\epsilon^{1/2}) \Brackets{\Norm{u_0}_{2,\Omega}+ \sqrt{\Norm{\gamma_ju_0}_{2, \Gamma_C}}},
	\]
	here the constant $C$ depends on $\Omega, \kappa_1, \kappa_2, \Delta, \Norm{N_l}_{W^{1,\infty}(Q)}$ and quality of the mesh (see \cite{Brenner2008} for detail description).
\end{theorem}
\begin{remark}
	For Robin problem, the conclusion is more elegant:
	
	Assume $u_0(\bm{x}) \in H^2(\Omega)$ and $N_l(\bm{y}) \in W^{1,\infty}_\sharp(Q)$. Neglect the error brought by calculating numerical correctors $N^\star_l(\bm{y})$ and homogenized coefficients $\hat{A}^\star$. Use Lagrange FEM to Solve \cref{eq:homogenized robin problem variation form} with $h$ grid size. The numerical error is:
		\[
		\Norm{\nabla u_\epsilon-(\nabla u_\epsilon)^\star}_{0,\Omega} \leq C (h+\epsilon^{1/2}) \Norm{u_0}_{2,\Omega},
		\]	
	here the constant $C$ depends on $\Omega, \kappa_1, \kappa_2, \alpha_1, \alpha_2, \Norm{N_l}_{W^{1,\infty}(Q)}$ and quality of the mesh.
\end{remark}

\section{Numerical Experiments}\label{sec:experiments}

Generally, when to solve contact problems described by hemivariational inequalities (for examples, \cref{eq:contact problem homohemiform} and \cref{eq:contact problem hemiform}), we need to convert the original problems to an optimization problem and use optimizaiton algorithm. Inspired by \cite{Barboteu2013}, we discover an iterative method to solve simplified problem \cref{eq:simplified contact problem}, and it performs effectively in our numerical experiment.

\begin{algorithm}[H]
	\caption{Iterative method for solving simplified contact problem \cref{eq:simplified contact problem}}
	\begin{algorithmic}[1]
		\State Mesh the domain $\Omega$ and construct FEM space $V_h$. Then set tolerance $\mathit{tol}$ and initial solutions $u^{(0)}_h$.
		\State Solve a variational problem: \label{al:1}
		\[
		\left\{
		\begin{aligned}
			&\text{Find } u^{(n+1)}_h \in V, \text{ s.t. } \forall v_h \in V_h \\
			&\int_\Omega A \nabla u^{(n+1)}_h \cdot \nabla v_h = \FuncAction{f}{v_h}+\int_{\Gamma_N} gv_h - \int_{\Gamma_C} \mathrm{B}\Brackets{u^{(n)}_h} v_h 
		\end{aligned}	
		\right. ,
		\]
		which is equivalent to solve a linear system:
		\[
		A_h u^{(n+1)} = b_h + r_h\Brackets{u^{(n)}} .
		\]
		\State Loop until {$\Norm{u^{(n+1)}_h - u^{(n)}} < \mathit{tol} \Norm{u^{(n)}}$}.
	\end{algorithmic}
\end{algorithm}

Actually, One can prove if \cref{pp:simplified contact problem} holds, the algorithm above will converge linearly to the real solution. The technique used here is to elucidate nonlinear map $u^{(n)} \mapsto u^{(n+1)}$ is contractive, and we omit the details here.

Now, we can set up our experiment problem:

Take $\Omega$ as square $(0,1)\times(0,1)$, and partition $\Omega$ into $N\times N$ whole cells. Hence, in this case $\epsilon = 1/N$. We set $A^\epsilon(\bm{x})=\kappa^\epsilon(\bm{x}) I$, $I$ is identity matrix, and $\kappa^\epsilon(\bm{x})$ can merely take two values respectively in different subdomains on each cell. Here a figure to illustrate those relations:
\begin{figure}[htbp]
	\centering
	\begin{tikzpicture}
		\draw[->] (-2.0, -1) -- (-1, -1);
		\draw[->] (-2.0, -1) -- (-2.0, 0);
		\node at (-0.7, -1) {$x$};
		\node at (-2.0, 0.3) {$y$};
		\draw (0, 0) rectangle (4, 4);
		\draw (0, 0) -- (0, -0.3) (2, 0) -- (2, -0.3) (4, 0) -- (4, -0.3);
		\draw[<-] (0, -0.15) -- (0.5, -0.15);
		\draw[<-] (2, -0.15) -- (2.5, -0.15);
		\draw[->] (1.5, -0.15) -- (2.0, -0.15);
		\draw[->] (3.5, -0.15) -- (4.0, -0.15);
		\draw (0, 0) -- (-0.3, 0) (0, 4) -- (-0.3, 4);
		\draw[<-] (-0.15, 4) -- (-0.15, 2.5);
		\draw[->] (-0.15, 1.5) -- (-0.15, 0);		
		\draw (4, 0) -- (4.3, 0) (4, 4) -- (4.3, 4);
		\draw[<-] (4.15, 4) -- (4.15, 2.5);
		\draw[->] (4.15, 1.5) -- (4.15, 0);
		\draw (0, 4) -- (0, 4.3) (4, 4) -- (4, 4.3);
		\draw[<-] (0, 4.15) -- (1.5, 4.15);
		\draw[->] (2.5, 4.15) -- (4, 4.15);
		\node at (1.0, -0.25) {$\Gamma_{C'}$};
		\node at (3.0, -0.25) {$\Gamma_{C''}$};
		\node at (-0.25, 2) {$\Gamma_N$};
		\node at (4.3, 2) {$\Gamma_N$};
		\node at (2.0, 4.25) {$\Gamma_D$};

		\foreach \m in {1,...,4}
			\foreach \n in {1,...,4}
			{
				\fill[mycolor1] (\m-1, \n-1) rectangle (\m, \n);
				\fill[mycolor2] (\m-0.75, \n-0.75) rectangle (\m-0.25, \n-0.25);
			}
		\foreach \i in {1,...,3}
		{
			\draw[dotted] (0, \i) -- (4, \i);
			\draw[dotted] (\i, 0) -- (\i, 4);
		}
	
	\fill[mycolor1] (7,0) rectangle (11, 4);
	\fill[mycolor2] (8,1) rectangle (10, 3);
	\draw[dotted] (7, -0.3) -- (7, 4.3) (6.7, 4) -- (11.3, 4) (11, 4.3) -- (11, -0.3) (11.3, 0.0) -- (6.7, 0);
	\draw[<-] (7, 2) -- (7.25, 2);
	\draw[->] (7.75, 2) -- (8, 2);
	\node at (7.5, 2) {$\rho$};
	\draw[<-] (9, 0) -- (9, 0.25);
	\draw[<-] (9, 1) -- (9, 0.75);
	\node at (9, 0.5) {$\rho$};
	\node at (9, 2) {$\kappa(\bm{y}) \equiv \kappa_2$};
	\node at (9, 3.5) {$\kappa(\bm{y}) \equiv \kappa_1$};
	
	\draw[mycolor3] (4.3, 2.6) -- (6.6, 3.6);
	\draw[mycolor3] (4.3, 1.4) -- (6.6, 0.4);
	\node[rounded corners=3pt, draw, fill=mycolor4] at (5.5, 2.0) {$\bm{y}=\bm{x}/\epsilon$};
	
	\node[rounded corners=3pt, draw, fill=mycolor4] at (2.0, 2.0) {$\Omega$};
	\end{tikzpicture}
	\caption{In this figure, $\Omega$ is composed by $4\times 4$ whole cells, a parameter $\rho$ indicates the geometric relation between the two subdomains of cell, here we set $\rho = 0.25$.}
\end{figure}
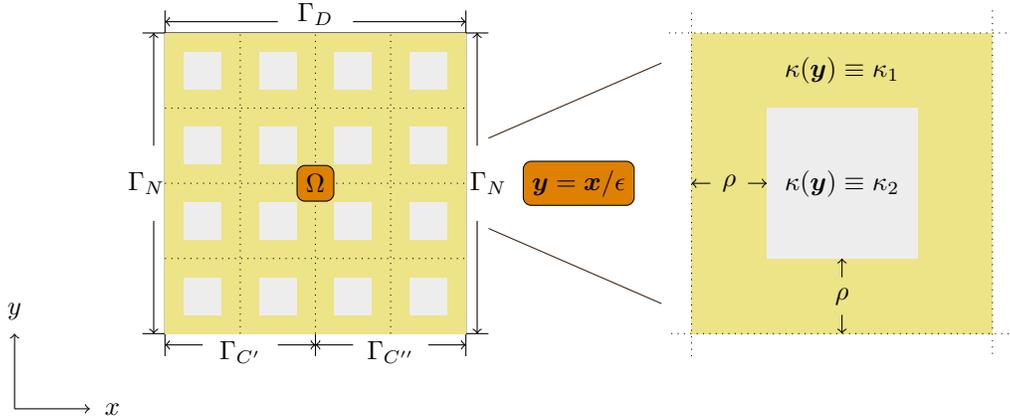

Boundary $\Gamma$ is divided into four parts $\Gamma_D, \Gamma_N, \Gamma_{C'}, \Gamma_{C''}$, and the boundary conditions are expressed in following variational problem:
\begin{equation}\label{eq:numerical experiment}
\left\{
\begin{aligned}
&\text{Find } u^\epsilon \in V= \{ v \in H^1(\Omega): v \equiv 0 \text{ on } \Gamma_D\}, \text{ s.t. } \forall v \in V \\
& \int_\Omega \kappa^\epsilon(\bm{x}) \nabla u^\epsilon \cdot \nabla v + \alpha \int_{\Gamma_{C'}} u^\epsilon v + \alpha \int_{\Gamma_{C''}} \Brackets{u^\epsilon}^+v = g \int_{\Gamma_N} v + f \int_{\Omega} v  
\end{aligned}
\right. .
\end{equation}
Here for simplicity, we set $\alpha, g, f$ as constants. To guarantee that this problem is solvable, we give following proposition:
\begin{proposition}
	If $\kappa_1 > \Seminorm{\alpha}$, then the solution of \cref{eq:numerical experiment} exists.
\end{proposition}
\begin{proof}
	According to \cref{ass:C}, we are left to show $c_j \leq 1$. Use the fact that $u(x, 1)\equiv 0$, we have
	\[
	\int_0^1 \Seminorm{u(x, 0)}^2 \dx x = \int_0^1 \Seminorm{\int_{0}^{1} \partial_y u(x,y) \dx y}^2 \dx x \leq \int_{0}^{1} \int_{0}^{1} \Seminorm{\partial_y u(x, y)}^2 \dx x \dx y ~ ,
	\]
	and this gives $c_j \leq 1$.
\end{proof}

We slice each cell equally into $M \times M$ elements, therefore, we actually solve the original problem \ref{eq:numerical experiment} and its homogenized version on a $NM\times NM$ grid. Our numerical experiment focus on verifying the homogenization error, we use following notation to measure the errors:
\[
\begin{aligned}
\mathbf{ERR_0} &\coloneqq \Norm{u_\epsilon - u_0}_{0, \Omega}/ \Norm{u_0}_{0,\Omega} \\
\mathbf{ERR_1} &\coloneqq \SquareBrackets{\sum_i\Norm{\partial_i u_\epsilon - \partial_i u_0 -(\partial_i N_l)^\epsilon \partial_l u_0}_{0, \Omega}^2}^{1/2} / \Seminorm{u_0}_{1, \Omega} \\
\mathbf{ERR_2} &\coloneqq \Seminorm{u_\epsilon - u_0}_{1,\Omega}/ \Seminorm{u_0}_{1,\Omega}
\end{aligned} ~~ .
\]
We set $\alpha=0.5, f=1.0, g=1.0, \kappa_1=1.0, \kappa_2 = 2.0$ and list the results in \cref{tab:1}.
\begin{table}[htbp] 
	\caption{The results of numerical experiments} \label{tab:1}
	\centering
	\begin{tabular}{cccc}
		\toprule
		~ & $\mathbf{ERR_0}$ & $\mathbf{ERR_1}$ & $\mathbf{ERR_2}$ \\
		\midrule
		$N=16, M=128$ & $0.00328$ & $0.00737$ & $0.21898$ \\
		$N=32, M=64$ & $0.00164$ & $0.00480$ & $0.21899$ \\
		$N=64, M=32$ & $0.00082$ & $0.00331$ & $0.21886$ \\
		$N=128, M=16$ & $0.00049$ & $0.00233$ & $0.21843$ \\
		Convergence rate & $0.92$ & $0.55$ & - \\
		\midrule
		$N=32, M=128$ & $0.00164$ & $0.00480$ & $0.21904$ \\
		\bottomrule
	\end{tabular}
\end{table}
From this table, the numerical convergence rate is actually close to its theoretical value $1.0$ and $0.5$, some differences may credit to that grid get coarser as heterogeneity or $1/\epsilon$ increase. In the last row of the table, we show the case $N=32, M=128$ as a compare to $N=32, M=64$, and the difference is few. This means that the $\mathbf{ERR_0}, \mathbf{ERR_1}$ and $\mathbf{ERR_2}$ we compute are accurate enough when $M$ is not small. Because of the limitation on computation resource, further and refined experiments such as $N=64, M=128$ or $N=128, M=128$ do not get conducted. We also notice that $\mathbf{ERR_2}$ do not decrease, This observation convinces the necessity of using First-order asymptotic solution.  

\section{Conclusions}
To model real scientific or engineering problems, Only studying the Dirichlet or Neumann boundary conditions is not completely adequate, and the situation has been encountered commonly in contact problems. Hence, the study on more suitable boundary condition is needed. A hemivariational inequalities framework for contact problems has been developed and also proved to be effective. Many physical and mechanical phenomena occur in highly heterogeneous media, and the simplified occasion is setting the coefficients of governing PDEs to have small periodicity. Contact problems in small periodicity setting have two major difficulties: one comes from nonlinearity in hemivariational inequalities, and the other originates from high oscillation due to multiscale property.

In this paper, several relatively reasonable assumptions are postulated to make the problems well posed, and a homogenization theorem is obtained by div-curl lemma. The key part is to derive $O(\epsilon^{1/2})$ estimation, and this result quantifies the convergence rate for first order expansion. Then, a computational method is proposed, and its numerical accuracy is also analyzed in FEM framework. We examine the special case--- Rubin problem and find out that an optimal $L^2$ estimation is obtainable.

It should be emphasized that, direct computational methods will cost enormous resources because of nonlinearity and high heterogeneity in this problem. It leads to the development of specialized computational methods. A thorough comparison of these two approaches and nontrivial numerical experiments will be more persuasive, and it will be provided in the future work.

\appendix

\begin{lemma}
	Let $\Omega$ be a Lipschitz domain and $\Gamma$ its boundary. Then there exists a constant $C(\Omega)$, such that $\forall \psi \in H^1(\Omega)$:
	\[
	C(\Omega) \Norm{\psi}_{1,\Omega}^2 \leq \int_{\Omega} \Seminorm{\nabla \psi}^2 + \int_{\Gamma} \psi^2 .
	\]
\end{lemma}
\begin{proof}
	If not, we have a sequence of $\{ \psi_n\}$ with $\Norm{\psi_n}_{1,\Omega}=1$ and $\int_{\Omega} \Seminorm{\nabla \psi_n}^2 + \int_{\Gamma} \psi_n^2 \leq 1/n$. Up to a subsequence, we will have:
	\[
	\left\{
	\begin{aligned}
	& \psi_n \rightarrow \psi_0 ~~~~ \text{ in } ~~ L^2(\Omega) \\
	& \nabla \psi_n \rightharpoonup \nabla \psi_0 ~~~~ \text{ in } ~~ L^2(\Omega)^d
	\end{aligned} ~~ .
	\right.
	\]
	Since $\int_{\Omega} \Seminorm{\nabla \psi_n}^2 \leq 1/n$, we obtain $\nabla \psi_n \rightarrow \bm{0}$ in $L^2(\Omega)^d$. We now have $\nabla \psi_0 \equiv \bm{0}$, and $\psi_0 \equiv C$. By $\int_\Omega \psi_0^2=\lim\limits_{n}\int_{\Omega} \psi_n^2 = 1-\lim\limits_{n}\int_{\Omega} \Seminorm{\nabla \psi_n}^2=1$, we know $C \neq 0$. Due to trace theorem, $\FuncAction{u}{\phi}=\int_{\Gamma}u\phi$ is a bounded functional on $H^1(\Omega)$ for any $u \in H^1(\Omega)$. Use weak convergence we have $\int_{\Gamma} \psi_0^2 = \lim\limits_{n} \int_{\Gamma} \psi_0 \psi_n \leq \Norm{\psi_0}_{0,\Gamma} \liminf_n \Norm{\psi_n}_{0,\Gamma}$, and 
	\[
	0 = \lim\limits_{n} \Seminorm{\nabla \psi_n}^2 + \int_{\Gamma} \psi_n^2 \geq \liminf_n \int_{\Gamma} \psi_n^2 \geq \int_{\Gamma} \psi_0^2 = \int_{\Gamma} C^2 ,
	\]
	and this contradicts $C\neq 0$.
\end{proof}

\begin{proposition}
	Suppose that $\Omega$ has $C^{1,1}$ boundary. In addition to uniformly ellipticity, coefficients $A(x)=\SquareBrackets{A_{ij}}_{1\leq i,j \leq d}$ are in $C^{0,1}(\Omega)$, and $\alpha(\bm{x})$ is $C^{0,1}(\Gamma)$ (in the sense of local coordinate). Then $\forall f \in L^2(\Omega)$, $u$ is the solution of Robin problem:
	\begin{equation}
	\left\{
	\begin{aligned}
	-\Div(A \nabla u) = f  \\
	\bm{n}\cdot A \nabla u + \alpha(\bm{x})u = 0
	\end{aligned} .
	\right.
	\end{equation}
	Then $u \in H^2(\Omega)$ with estimation $\Norm{u}_{2,\Omega} \leq C \Norm{f}_{0,\Omega}$, here $C$ depends on $\Omega, \Norm{A_{ij}}_{C^{0,1}(\Omega)}, \Norm{\alpha}_{C^{0,1}(\Gamma)}$ and $\kappa_1, \kappa_2$.	
\end{proposition}
\begin{proof}
	By flatting boundary technique, we only need to consider a half sphere $B_1^+$ in $\R^d_+$ as domain, and plane $T=\{\bm{x}\in \R^d \colon x_d=0\}$ as boundary where Robin condition imposed. The original problem can locally transfer to following form:
	\[
	\int_{B_1^+} A \nabla u \cdot \nabla v + \int_T \alpha u v = \int_{B_1^+} f v ~~~~ \forall v \in C^\infty(B_1^+) \text{ s.t. } \mathrm{dist}\Brackets{\Supp(v), \partial B_1^+ \cap \R^d_+ } > 0 .
	\]
	Take differential quotient $\Delta^h=\Delta^h_k$, and by a similar process we will get a result as in \cite{Gilbarg1983} sect. 8.3, here $k\neq d$:
	\[
	\begin{aligned}
		&\int_{B_1^+} A(\bm{x}+h\bm{e}_k) \nabla(\Delta ^h u) \cdot \nabla v + \int_T \alpha(\bm{x}+h\bm{e}_k) (\Delta ^h u) v \\
		\leq & \Brackets{\Norm{f}_{0, B_1^+}+\Norm{A}_{C^{0,1}(B_1^+)}\Norm{\nabla u}_{0, B_1^+}} \Norm{\nabla v}_{0, B_1^+} + \Norm{\alpha}_{C^{0,1}(T)}\Norm{u}_{0, T}\Norm{v}_{0, T}
	\end{aligned}.
	\]
	Take $\eta$ as cut-off function which $\eta \equiv 1$ in $B_{1/2}^+$ and $\eta \equiv 0$ in $B_{1}^+\setminus B_{3/4}^+$, then substitute $\eta^2 \Delta ^h u$ for $v$ into above inequality. We have:
	\[
		\int_{B_{1}^+} \Seminorm{\eta \nabla \Delta ^h u}^2 + \int_T \Seminorm{\eta \Delta ^h u}^2 \leq C \Brackets{\Norm{u}_{1, B_{1}^+}+\Norm{f}_{0, B_{1}^+}}.
	\]
	This implies $\Norm{\partial_{ij}u}_{0, B_{1/2}^+} \leq C\Brackets{\Norm{u}_{1, B_{1}^+}+\Norm{f}_{0, B_{1}^+}}$ for $i \neq d$ or $j \neq d$. Recall 
	\[
		f = -a_{dd} \partial_{dd} u - \sum_{i \neq d \text{ or } j \neq d} a_{ij} \partial_{ij} u - \sum_{ij} \partial_i a_{ij} \partial_j u
	\]
	It then follows $\Norm{\partial_{dd}u}_{0, B_{1/2}^+} \leq C\Brackets{\Norm{u}_{1, B_{1}^+}+\Norm{f}_{0, B_{1}^+}}$. Finally, we have $\Norm{u}_{2, B_{1/2}^+} \leq C \Norm{f}_{0, B_{1}^+}$.
\end{proof}

%\section{Novelty}
%The first result (homogenization) in our paper improves the Z. Liu's work in "Homogenization of boundary hemivariational inequalities in linear elasticity" with suitable assumptions. To our knowledge, the second result (homogenization error estimation) and numerical analysis are completely new, while some proof techniques are adopted from Shen's paper "Convergence rates in periodic homogenization of systems of elasticity".

\bibliographystyle{plain}
\bibliography{references}

\end{document}